\documentclass[11pt,a4paper]{article}

\usepackage{titlesec}
\usepackage{fancyhdr}
\usepackage{a4wide}
\usepackage{graphicx}
\usepackage{float}
\usepackage{amssymb}
\usepackage{amsmath}
\usepackage{amsfonts}
\usepackage{amsthm}
\usepackage{color}
\usepackage{mathrsfs}
\usepackage{array}
\usepackage{eucal}
\usepackage{tikz}
\usepackage[T1]{fontenc}
\usepackage{inputenc}
\usepackage[english]{babel}
\usepackage{lmodern}
\usepackage{hyperref}
\usepackage{geometry}
\usepackage{changepage}
\usepackage{bm}
\changepage{0pt}{}{}{}{}{0pt}{}{0pt}{6pt}
\usepackage[numbers]{natbib}
\usepackage{hyperref}
\hypersetup{
pdfpagemode=none,
pdftoolbar=true,        
pdfmenubar=true,        
pdffitwindow=false,     
pdfstartview={Fit},    
pdftitle={Design of an acoustic energy distributor using thin resonant slits},    
pdfauthor={L. Chesnel, S.A. Nazarov},     
pdfsubject={},  
pdfcreator={L. Chesnel, S.A. Nazarov},   
pdfproducer={L. Chesnel, S.A. Nazarov}, 
pdfkeywords={}, 
pdfnewwindow=true,      
colorlinks=true,       
linkcolor=magenta,          
citecolor=red,        
filecolor=cyan,      
urlcolor=blue           
}

\newcommand{\dsp}{\displaystyle}

\newcommand{\eps}{\varepsilon}
\newcommand{\om}{\omega}
\newcommand{\Om}{\Omega}
\newcommand{\mrm}[1]{\mathrm{#1}}

\newcommand{\Cplx}{\mathbb{C}}
\newcommand{\N}{\mathbb{N}}
\newcommand{\R}{\mathbb{R}}

\newtheorem{theorem}{Theorem}[section]
\newtheorem{lemma}[theorem]{Lemma}
\newtheorem{remark}[theorem]{Remark}

\begin{document}

~\vspace{0.0cm}
\begin{center}
{\sc \bf\huge Design of an acoustic energy distributor\\[6pt] using thin resonant slits}
\end{center}

\begin{center}
\textsc{Lucas Chesnel}$^1$, \textsc{Sergei A. Nazarov}$^{2}$\\[16pt]
\begin{minipage}{0.95\textwidth}
{\small
$^1$ INRIA/Centre de math\'ematiques appliqu\'ees, \'Ecole Polytechnique, Institut Polytechnique de Paris, Route de Saclay, 91128 Palaiseau, France;\\
$^2$ Institute of Problems of Mechanical Engineering, Russian Academy of Sciences, V.O., Bolshoj pr., 61, St. Petersburg, 199178, Russia;\\[10pt]
E-mails: \texttt{lucas.chesnel@inria.fr}, \texttt{srgnazarov@yahoo.co.uk}, \texttt{s.nazarov@spbu.ru} \\[-14pt]

\begin{center}
(\today)
\end{center}
}
\end{minipage}
\end{center}
\vspace{0.4cm}

\noindent\textbf{Abstract.} 
We consider the propagation of time harmonic acoustic waves in a device made of three unbounded channels connected by thin slits. The wave number is chosen such that only one mode can propagate. The main goal of this work is to present a device which can serve as an energy distributor. More precisely, the geometry is first designed so that for an incident wave coming from one channel, the energy is almost completely transmitted in the two other channels. Additionally, adjusting slightly two geometrical parameters, we can control the ratio of energy transmitted in the two channels. The approach is based on asymptotic analysis for thin slits around resonance lengths. We also provide numerical results to illustrate the theory.\\


\section{Introduction}\label{Introduction}

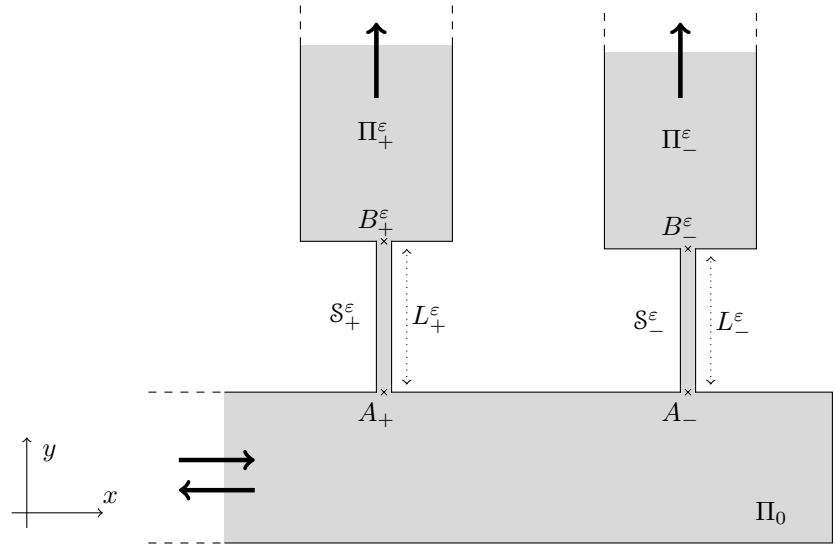
\begin{figure}[!ht]
\centering
\begin{tikzpicture}[scale=2]
\draw[fill=gray!30,draw=none](-4,0) rectangle (0,1);
\draw (-4,0)--(0,0)--(0,1)--(-4,1);
\draw[dashed] (-4.5,0)--(-4,0);
\draw[dashed] (-4.5,1)--(-4,1);
\node at (-0.4,0.2){\small $\Pi_0$};
\begin{scope}[xshift=-2cm,yshift=0cm]
\draw[fill=gray!30,draw=none](-1.5,3.3) rectangle (-0.5,2);
\draw (-1.5,3.3)--(-1.5,2)--(-0.5,2)--(-0.5,3.3);
\draw[dashed] (-1.5,3.3)--(-1.5,3.6);
\draw[dashed] (-0.5,3.3)--(-0.5,3.6);
\draw[fill=gray!30,draw=black](-1,1) rectangle (-0.9,2);
\draw[fill=gray!30,draw=none](-0.995,0.9) rectangle (-0.905,2.1);
\draw (-0.97,0.98)--(-0.93,1.02);
\draw (-0.97,1.02)--(-0.93,0.98);
\node at (-1.2,1.5){\small $\mathcal{S}^{\eps}_{+}$};
\draw[dotted,<->] (-0.8,1.05)--(-0.8,1.95);
\node at (-0.65,1.48){\small $L^{\eps}_{+}$};
\begin{scope}[yshift=1cm]
\draw (-0.97,0.98)--(-0.93,1.02);
\draw (-0.97,1.02)--(-0.93,0.98);
\end{scope}
\node at (-1,0.85){\small $A_{+}$};
\node at (-1,2.12){\small $B^{\eps}_{+}$};
\node at (-1,2.7){\small $\Pi^{\eps}_+$};
\end{scope}
\begin{scope}[xshift=0cm,yshift=0cm]
\begin{scope}[xshift=0cm,yshift=-0.05cm]
\draw[fill=gray!30,draw=none](-1.5,3.3) rectangle (-0.5,2);
\draw (-1.5,3.3)--(-1.5,2)--(-0.5,2)--(-0.5,3.3);
\draw[dashed] (-1.5,3.3)--(-1.5,3.6);
\draw[dashed] (-0.5,3.3)--(-0.5,3.6);
\draw[fill=gray!30,draw=black](-1,1.05) rectangle (-0.9,2);
\draw[fill=gray!30,draw=none](-0.995,0.9) rectangle (-0.905,2.1);
\node at (-1,2.12){\small $B^{\eps}_{-}$};
\node at (-1,2.7){\small $\Pi^{\eps}_-$};
\draw[->,line width=0.6mm] (-4.3,0.6)--(-3.8,0.6);
\draw[<-,line width=0.6mm] (-4.3,0.4)--(-3.8,0.4);
\draw[->,line width=0.6mm] (-1,3)--(-1,3.5);
\draw[->,line width=0.6mm] (-3,3)--(-3,3.5);
\begin{scope}[yshift=1cm]
\draw (-0.97,0.98)--(-0.93,1.02);
\draw (-0.97,1.02)--(-0.93,0.98);
\end{scope}
\end{scope}
\draw (-0.97,0.98)--(-0.93,1.02);
\draw (-0.97,1.02)--(-0.93,0.98);
\node at (-1.2,1.45){\small $\mathcal{S}^{\eps}_{-}$};
\draw[dotted,<->] (-0.8,1.05)--(-0.8,1.9);
\node at (-0.65,1.44){\small $L^{\eps}_{-}$};
\node at (-1,0.85){\small $A_{-}$};
\end{scope}
\begin{scope}[shift={(-5.4,0)}]
\draw[->] (0,0.2)--(0.6,0.2);
\draw[->] (0.1,0.1)--(0.1,0.7);
\node at (0.65,0.3){\small $x$};
\node at (0.25,0.6){\small $y$};
\end{scope}
\end{tikzpicture}
\caption{Geometry of the waveguide $\Om^{\eps}$. \label{Waveguide}} 
\end{figure}

In this article, we are interested in the design of an acoustic energy distributor. More precisely, we study the propagation of time harmonic waves at a given wave number in a structure with three unbounded channels. One of them plays the role of input/output channel while the two others are only output channels (see Figure \ref{Waveguide}). Our goal is to find a geometry where the energy of an incoming wave is almost completely transmitted and additionally where we can control the ratio of energy transmitted in the two other output channels. The main difficulty of this problem lies in the fact that the dependence of the acoustic field with respect to the geometry is nonlinear and implicit. A device similar to the one that we wish to create is the acoustic power divider \cite{LaLI05,AnEl05,NgCa07,EsBA20}. The difference is that in our case, we want to be able to control the ratio of energy transmitted in the two output channels. We mention that such devices are very interesting for applications not only in acoustics but also for example in optics \cite{PLLPBV15,FESJY13} or in radio electronics \cite{ChWu11} (see the literature concerning radio diplexers). \\
\newline
To construct such particular waveguides, it has been proposed to work with so-called zero-index materials or similar metamaterials, see for example \cite{EsBA20} and the references therein. However, from our understanding, these materials are still hard to handle in practice. In this article, we propose a different approach relying on the use of a classical medium but with a well-chosen shape. More precisely, we will work with thin slits as illustrated in Figure \ref{Waveguide}. In general, due to the geometrical features, almost no energy passes through the slits and it may seem a bit paradoxical to use them to have almost complete transmission. However working around the resonance lengths, it has been shown that we can observe this phenomenon. This has been studied for example in \cite{Krieg,BoTr10,LiZh17,LiZh18,LiSZ19} in the context of the scattering of an incident wave by a periodic array of subwavelength slits. The approach that we will consider is based on matched asymptotic expansions. For related techniques, we refer the reader to \cite{Beal73,Gady93,KoMM94,Naza96,Gady05,Naza05,JoTo06,BaNa15,BoCN18}. We emphasize that an important feature of our work distinguishing it from the previous references is that the lengths, and not only the widths, of the slits depend on $\eps$ (see (\ref{DefL})). This way of considering the problem is an essential ingredient of the analysis. From this respect, our work shares similarities with \cite{HoSc19,BrHS20,BrSc20} (see also references therein). The difference is that we combine several thin slits and coupling effect can appear.\\
\newline
Let us mention that techniques of optimization (see e.g. \cite{LDOHG19,LGHDO19,Lebb19}) have been applied to exhibit energy acoustic distributors. However they involve non convex functionals and unsatisfactory local minima exist. Moreover, they offer no control on the obtained shape compare to the approach we propose here. In particular, with our geometry, a small change of the geometry allows us to transmit the energy in one channel instead of the other. In the context of propagation of acoustic waves, another device which is interesting in practice is the modal converter \cite{HoHu05,CXJD06,KTSM09}. At higher frequency, when several modes can propagate, the aim is to have a structure where the energy of an incident mode is transferred onto another mode. We do not know if thin slits can be useful to obtain such effect.\\
\newline 
The outline is as follows. In the next section, we present the geometry and the notation. Then in Section \ref{SectionAux}, we introduce two auxiliary problems which will be involved in the analysis. The Section \ref{SectionAsympto} constitutes the heart of the article: here we compute an asymptotic expansion of the acoustic field and of the scattering coefficients with respect to $\eps$, the width of the thin slits. Then we exploit the results in Section \ref{AnalysisResults} to exhibit situations where the device acts as an energy distributor. We illustrate the theory in Section \ref{SectionNumerics} with numerical experiments before discussing possible extensions and open
questions in Section \ref{SectionConclusion}. Finally we give the proof of two technical lemmas needed in the study in a short appendix.

\section{Setting}

First, we describe in detail the geometry (see Figure \ref{Waveguide}). Set $\Pi_0:=\{z=(x,y)\in\R^2\,|\,(x,y)\in(-\infty;0)\times(0;1)\}$. Pick two different points $A_{\pm}=(p_\pm,1)\in(-\infty;0)\times\{1\}\subset\partial\Pi_0$. For $\eps>0$, define the lengths
\begin{equation}\label{DefL}
L^\eps_{\pm}:=L_{\pm}+\eps L^\prime_{\pm}
\end{equation}
where the values $L_{\pm}>0$, $L^\prime_\pm>0$ will be fixed later on to observe interesting phenomena. Define the thin strips 
\[
\mathcal{S}^{\eps}_{\pm}:=(p_{\pm}-\eps/2;p_{\pm}+\eps/2)\times[1;1+L_{\pm}^{\eps}].
\]
Define the points $B^{\eps}_{\pm}$ such that $B^{\eps}_{\pm}=(p_{\pm},1+L_{\pm}^{\eps})$. Set 
\[
\Pi^{\eps}_{\pm}:=(p_{\pm}-1/2;p_{\pm}+1/2)\times(1+L_{\pm}^{\eps};+\infty).
\]
And finally, we define the geometry
\[
\Om^{\eps}:=\Pi_0\cup\mathcal{S}^{\eps}_{+}\cup\mathcal{S}^{\eps}_{-}\cup \Pi^{\eps}_{+}\cup \Pi^{\eps}_{-}.
\]
Interpreting the domain $\Om^{\eps}$ as an acoustic waveguide, we are led to consider the following problem with Neumann boundary condition
\begin{equation}\label{MainPb}
 \begin{array}{|rcll}
 \Delta u^\eps +\omega^2  u^\eps&=&0&\mbox{ in }\Omega^\eps\\
 \partial_\nu u^\eps &=&0 &\mbox{ on }\partial\Omega^\eps.
\end{array}
\end{equation}
Here, $\Delta$ is the Laplace operator while $\partial_\nu$ corresponds to the derivative along the exterior normal. Furthermore, $u^\eps$ is the acoustic pressure of the medium while $\omega>0$ is the wave number of the plane modes $\mrm{w}_{h}^\pm(x,y)=e^{\pm i\omega x}$ (resp. $\mrm{w}_{v}^\pm(x,y)=e^{\pm i\omega y}$) propagating in $\Pi_0$ (resp. $\Pi^{\eps}_{\pm}$). We fix $\om\in(0;\pi)$ so that no other mode can propagate. We are interested in the solution to the diffraction problem \eqref{MainPb} generated by the incoming wave $\mrm{w}^+_{h}$ in the trunk $\Pi_0$. This solution admits the decomposition
\begin{equation}\label{Field}
 u^\eps(x,y)=\begin{array}{|ll}
\mrm{w}^+_h(x)+R^\eps\,\mrm{w}^-_h(x)+\dots \quad\mbox{ in }\Pi_0\\[3pt]
\phantom{\mrm{w}^+_h(x)\,\,\quad} T^\eps_+\,\mrm{w}^+_v(y-L^{\eps})+\dots \quad\mbox{ in }\Pi_{+}^\eps\\[3pt]
\phantom{\mrm{w}^+_h(x)\,\,\quad} T^\eps_-\,\mrm{w}^+_v(y-L^{\eps})+\dots \quad\mbox{ in }\Pi_{-}^\eps
 \end{array}
\end{equation}
where $R^\eps\in\mathbb{C}$ is a reflection coefficient and $T^\eps_{\pm}\in\mathbb{C}$ are transmission coefficients. In this decomposition, the ellipsis stand for a remainder which decays at infinity with the rate $e^{-(\pi^2-\om^2)^{1/2}|x|}$ in $\Pi_0$ and $e^{-(\pi^2-\om^2)^{1/2}|y|}$ in $\Pi_{\pm}^\eps$. Due to conservation of energy, one has
\[
|R^\eps|^2+|T_+^\eps|^2+|T_-^\eps|^2=1.
\]
In general, almost no energy of the incident wave $\mrm{w}^+_h$ passes through the thin strips and one observes almost complete reflection. More precisely, one finds that there holds 
\begin{equation}\label{DefExpanNonCri}
 R^\eps  =1+\tilde{R}^{\,\eps},\qquad\qquad T^\eps_{\pm}  =\tilde{T}^{\,\eps}_{\pm},
\end{equation}
 where  $\tilde{R}^{\eps}$, $\tilde{T}^{\eps}_{\pm}$ tend to zero as $\eps$ goes to zero. The main goal of this work is to show that choosing carefully the lengths $L^{\eps}_{\pm}$ of the thin strips $\mathcal{S}^{\eps}_{\pm}$ as well as their positions, the energy of the wave $\mrm{w}^+_h$ can be almost completely transmitted. Moreover we can control the energy transmitted respectively in $T_+^\eps$ and $T_-^\eps$. More precisely, we will prove that choosing carefully $L^{\eps}_{\pm}$, as $\eps$ tends to zero we can have
 \[
R^\eps
 =\tilde{R}^{\,\eps},\qquad\qquad T^\eps_{\pm} =T^0_{\pm}+\tilde{T}^{\,\eps}, 
\]
where $\tilde{R}^{\eps}$, $\tilde{T}^{\eps}_{\pm}$ tend to zero as $\eps$ goes to zero, $|T^0_{+}|^2+|T^0_{-}|^2=1$ and $|T^0_{+}|/|T^0_{-}|$ can be any number in $(0;+\infty)$ (see formulas \eqref{FinalDistributor} below). Thus we can select the energy ratio transmitted in $\Pi_{\pm}^\eps$ and the device acts as an energy distributor. 
\section{Auxiliary objects} \label{SectionAux}
In this section, we discuss a couple of boundary value problems whose solutions will appear in the construction of the asymptotic expansions of the acoustic field $u^{\eps}$.\\
\newline
Considering the limit $\eps\rightarrow0^+$ in the equation \eqref{MainPb} restricted to the strips $\mathcal{S}^{\eps}_{\pm}$, we are led to study the one-dimensional Helmholtz equations
\begin{equation}\label{Pb1D1}
\partial^2_{y}v+\omega^2v=0\qquad\mbox{ in }\ell_{\pm}:=(1;1+L_{\pm})
\end{equation}
supplied with the artificially imposed Dirichlet conditions
\begin{equation}\label{Pb1D2}
  v(1)=v(1+L_{\pm})=0.
\end{equation}
Eigenvalues and eigenfunctions (up to a multiplicative constant) of the boundary value problem \eqref{Pb1D1}--\eqref{Pb1D2} are given by
\[
\dsp\mu_{m}:=(\pi m/L_{\pm})^2,\qquad v_{m}(y)=\sin(\pi m(y-1)/L_{\pm}),
\]
with $m\in\N^{\ast}:=\{1,2,3,\dots\}$. \\

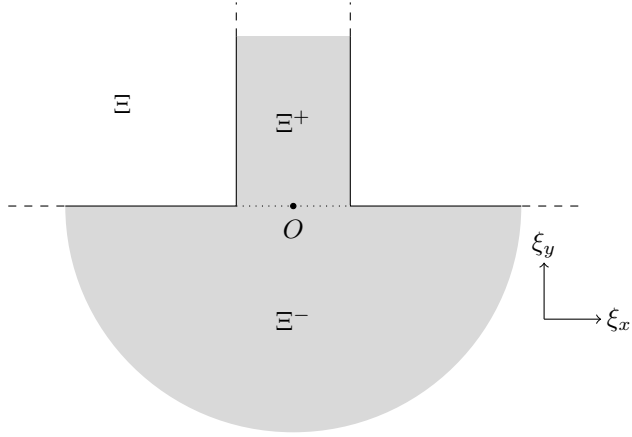
\begin{figure}[!ht]
\centering
\begin{tikzpicture}[scale=1.5]
\begin{scope}
\clip (-2,2) -- (2,2) -- (2,0) -- (-2,0) -- cycle;
\draw[fill=gray!30,draw=none](0,2) circle (2);
\end{scope}
\draw[fill=gray!30,draw=none](-0.5,2) rectangle (0.5,3.5);
\draw (-2,2)--(-0.5,2)--(-0.5,3.5);
\draw (2,2)--(0.5,2)--(0.5,3.5);
\draw[dashed] (-0.5,3.5)--(-0.5,3.8);
\draw[dashed] (0.5,3.5)--(0.5,3.8);
\draw[dashed] (2.5,2)--(2,2);
\draw[dashed] (-2.5,2)--(-2,2);
\draw[thick,draw=gray!30](-0.495,2)--(0.495,2);
\node[mark size=1pt,color=black] at (0,2) {\pgfuseplotmark{*}};
\node[color=black] at (0,1.8) {\small $O$};
\begin{scope}[shift={(2.2,1)}]
\draw[->] (0,0)--(0.5,0);
\draw[->] (0,0)--(0,0.5);
\node at (0.66,0){\small$\xi_x$};
\node at (0,0.66){\small$\xi_y$};
\end{scope}
\node at (0,1){\small$\Xi^-$};
\node at (0,2.75){\small$\Xi^+$};
\node at (-1.5,2.9){\small$\Xi$}; 
\node at (2.5,2.9){\small\phantom{$\Xi=\R^2_-\cup\Pi\cup(-1/2;1/2)\times\{0\}$}};
\draw[dotted] (-0.5,2)--(0.5,2);
\end{tikzpicture}
\caption{Geometry of the frozen domain $\Xi$.\label{FrozenGeom}} 
\end{figure}

\noindent Now we present a second problem which is involved in the construction of asymptotics and which will be used to describe the boundary layer phenomenon near the points $A_{\pm}$, $B^{\eps}_{\pm}$. To capture rapid variations of the field for example in the vicinity of $A_{\pm}$, we introduce the stretched coordinates $\xi^\pm=(\xi^\pm_x,\xi^\pm_y)=\eps^{-1}(z-A_\pm)=(\eps^{-1}(x-p_{\pm}),\,\eps^{-1}(y-1))$. Observing that 
\begin{equation}\label{Strechted1}
(\Delta_z+\om^2)u^{\eps}(\eps^{-1}(z-A_\pm))=\eps^{-2}\Delta_{\xi^{\pm}}u^{\eps}(\xi^{\pm})+\dots,
\end{equation}
we are led to consider the Neumann problem 
\begin{equation}\label{PbBoundaryLayer}
-\Delta_\xi Y=0\qquad\mbox{ in }\Xi,\quad\qquad \partial_\nu Y=0\quad\mbox{ on }\partial\Xi,
\end{equation}
where $\Xi:=\Xi^-\cup\Xi^+\subset\mathbb{R}^2$ (see Figure \ref{FrozenGeom}) is the union of the half-plane $\Xi^-$ and the semi-strip $\Xi^+$ such that
\[
\Xi^-:={\mathbb R}^2_-=\{\xi=(\xi_x,\xi_y):\,\xi_y<0\},\qquad\qquad
\Xi^+:=\{\xi:\,\xi_y\geq 0, |\xi_x|<1/2\}.
\]
In the method of matched asymptotic expansions (see
the monographs \cite{VD,Ilin}, \cite[Chpt. 2]{MaNaPl} and others) that we will use, we will work with solutions of \eqref{PbBoundaryLayer} which are bounded or which have polynomial growth in the semi-strip as $\xi_y\rightarrow+\infty$. One of such solutions is evident and is given by $Y^0=1$.
Another solution, which is linearly independent with $Y^0$, is the unique function satisfying (\ref{PbBoundaryLayer}) and which has the representation
\begin{equation}\label{PolyGrowth}
Y^1(\xi)=\left\{
\begin{array}{ll}
\xi_y+C_\Xi+O(e^{-\pi \xi_y})& \mbox{ as }\xi_y\rightarrow+\infty,\quad \xi
\in \Xi^+\\[5pt]
\dsp\frac{1}{\pi}\ln \frac{1}{|\xi|}+O
 \Big(\frac{1}{|\xi|}\Big)& \mbox{ as }|\xi|\rightarrow+\infty,\quad \xi
\in \Xi^-.
\end{array}\right.
\end{equation}
Here, $C_\Xi$ is a universal constant whose value can be computed using conformal mapping, see for example \cite{Schn17}. Note that the coefficients in front of the growing terms in \eqref{PolyGrowth} are related due to the fact that a harmonic function has zero total flux at infinity.

\section{Asymptotic analysis}\label{SectionAsympto}
In this section, we compute the asymptotic expansion of the field $u^{\eps}$ in (\ref{Field}) as $\eps$ tends to zero. The final results are summarized in (\ref{AsymptoFinalResults}). We assume that the limit lengths $L_{\pm}$ of the thin strips (see (\ref{DefL})) are such that
\begin{equation}\label{ConditionLongueur}
L_{\pm} = \frac{\pi m_{\pm}}{\om} \qquad\mbox{for some}\,\,m_{\pm}\in\N^{\ast}.
\end{equation}
In other words, we assume that $\omega^2$ is an eigenvalue of the problems \eqref{Pb1D1}--\eqref{Pb1D2}. We emphasize that these problems are posed in the fixed lines $\ell_{\pm}$ but the true lengths $L^\eps_{\pm}=L_{\pm}+\eps L'_{\pm}$ of the strips $\mathcal{S}^\eps_{\pm}$ depend on the parameter $\eps$. In the channels, we work with the ansatz
\begin{equation}\label{AnsatzWaveguides}
u^\eps=u^0_0+\eps u^\prime _0+\dots \quad\mbox{\rm in }\Pi_0,\qquad
u^\eps=u^0_\pm+\eps u^\prime _\pm+\dots \quad\mbox{\rm in }\Pi_\pm^{\eps},
\end{equation}
while in the thin strips, we deal with the expansion 
\[
u^\eps(x,y)=\eps^{-1}v^{-1}_{\pm}(y)+v^0_{\pm}(y)+\dots\quad\mbox{\rm in }\mathcal{S}^\eps_{\pm}.
\]
Taking the formal limit $\eps\to0^+$, we find that $v^{-1}_{\pm}$ must solve the homogeneous problem (\ref{Pb1D1})--(\ref{Pb1D2}). Under the assumption (\ref{ConditionLongueur}) for the lengths $L_{\pm}$, non zero solutions exist for this problem and we look for $v^{-1}_{\pm}$ in the form 
\[
v^{-1}_{\pm}(y)=a_{\pm}{\bf v}_{\pm}(y)\qquad\mbox{ with }\qquad a_\pm\in\mathbb{C},\ {\bf v}_{\pm}(y)=\sin(\omega (y-1)).
\]  
Let us stress that the values of $a_{\pm}$ are unknown and will be fixed during the construction of the asymptotics of $u^\eps$. 
At $A_\pm$, the Taylor formula gives
\begin{equation}\label{DefInnerA}
\eps^{-1}v^{-1}_{\pm}(y)+v^0_{\pm}(y)=0+(C_\pm^A \xi^\pm_y  +v^0_{\pm}(1))+\dots\qquad\mbox{ with }\qquad C_\pm^{A}:=a_{\pm}\partial_y {\bf v}(1)=a_{\pm}\omega.
\end{equation}
Here $\xi^\pm_y=\eps^{-1}(y-1)$ is the stretched variable introduced just before (\ref{Strechted1}). At $B_\pm^{\eps}$, we have
\begin{equation}\label{DefInnerB}
\eps^{-1}v^{-1}_{\pm}(y)+v^0_{\pm}(y)=0+(C_\pm^B \zeta^\pm_y  +v^0_{\pm}(1+L_{\pm}))+\dots
\end{equation}
with 
\begin{equation}\label{DefInnerBbis}
C_\pm^B:=-a_{\pm}\partial_y {\bf v}(1+L_{\pm})=-a_{\pm}\omega\cos(\om L_{\pm})=(-1)^{1+m_{\pm}}a_{\pm}\omega.
\end{equation}
Here, we use the stretched coordinates $\zeta^\pm=(\zeta^\pm_x,\zeta^\pm_y)=(\eps^{-1}(x-p_{\pm}),\,\eps^{-1}(1+L_{\pm}-y))$ (mind the sign of $\zeta^\pm_y$).\\
\newline
We look for an inner expansion of $u^{\eps}$ in the vicinity of $A_{\pm}$ of the form
\[
u^\eps(x)=C^A_\pm\,Y^1(\xi^\pm)+ c^A_\pm+\dots
\]
where $Y^1$ is introduced in (\ref{PolyGrowth}), $C^A_\pm$ are defined in (\ref{DefInnerA}) and $c^A_\pm$ are constants to determine. In a vicinity of $B^{\eps}_{\pm}$, we look for an inner expansion of $u^{\eps}$ of the form 
\[
u^\eps(x)=C^B_\pm\,Y^1(\zeta^\pm_x,\zeta^\pm_y+L'_{\pm})+ c^B_\pm+\dots
\]
where $C^B_\pm$ are defined in (\ref{DefInnerBbis}) and $c^B_\pm$ are constants to determine.\\
\newline 
Let us continue the matching procedure. Taking the limit $\eps\to0^+$, we find that the main term $u^0_0$ in (\ref{AnsatzWaveguides}) must solve the problem
\[
\Delta u^0_0 +\omega^2u^0_0=0\ \mbox{ in }\Pi_0,\qquad 
\partial_\nu u^0_0=0\mbox{ on }\partial\Pi_0\setminus \{A_+,A_-\},
\]
with the expansion 
\[
u^0_0=\mrm{w}^+_h+R^0\,\mrm{w}^-_h+\tilde{u}^0_0.
\]
Here $R^0\in\Cplx$ and $\tilde{u}^0_0$ decay exponentially at infinity. Moreover, we find that the term $u^0_\pm$ in (\ref{AnsatzWaveguides}) must solve the problem
\[
\Delta u^0_\pm +\omega^2u^0_\pm=0\ \mbox{ in }\Pi^{\eps}_{\pm},\qquad 
\partial_\nu u^0_\pm=0\mbox{ on }\partial\Pi^{\eps}_{\pm}\setminus B^{\eps}_{\pm},
\]
with the expansion 
\[
u^0_{\pm}(x,y)=T^0_{\pm}\,\mrm{w}^+_v(y-L^{\eps})+\tilde{u}^0_{\pm}(x,y).
\]
Here $T^0_{\pm}\in\Cplx$ and $\tilde{u}^0_{\pm}$ decay exponentially at infinity. The coefficients $R^0$, $T^0_{\pm}\in\mathbb{C}$ will provide the first terms in the asymptotics of $R^{\eps}$, $T^{\eps}_{\pm}$:
\[
R^{\eps}=R^0+\dots\qquad\mbox{ and }\qquad T^{\eps}_{\pm}=T^0_{\pm}+\dots.
\]
Matching the behaviours of the inner and outer expansions of $u^{\eps}$ in $\Pi_0$, we find that at the points $A^{\pm}$, the function $u^0_{0}$ must expand as 
\[
u^0_0(x,y)= C^A_\pm\frac{1}{\pi}\ln \frac{1}{r^{A}_\pm}+U^0_{0\pm}+O(r^{A}_\pm)
\qquad \mbox{ as }r^{A}_\pm:=((x- p_{\pm})^2+(y-1)^2)^{1/2}\rightarrow0^+,
\]
where $U^0_0$ is a constant. Observe that $u^0_0$ is singular both at $A^{+}$ and $A^{-}$. Integrating by parts in 
\[
0=\int_{\Pi_{0,\rho}}(\Delta u^0_0 +\omega^2u^0_0)(e^{+i\omega x} + e^{-i\omega x})-u^0_0\,(\Delta (e^{+i\omega x} + e^{-i\omega x}) +\omega^2(e^{+i\omega x} + e^{-i\omega x})\,dxdy,
\]
with $\Pi_{0,\rho}:=\{(x,y)\in\Pi_0\,,x>-\rho\mbox{ and }r^A_{\pm}>1/\rho\}$, and taking the limit $\rho\to+\infty$, we get $2i\om(R^0-1)+2C^A_+\cos(\om p_+)+2C^A_-\cos(\om p_-)=0$. From the expressions of $C^A_{\pm}$ (see (\ref{DefInnerA})), this gives
\begin{equation}\label{equation1}
R^0=1+i(a_+\cos(\om p_+)+a_-\cos(\om p_-)).
\end{equation}
Then matching the behaviours of the inner and outer expansions of $u^{\eps}$ in $\Pi^{\eps}_\pm$, we find that at the points $B^\eps_{\pm}$, the function $u^0_\pm$ must expand as 
\[
u^0_\pm(x,y)= C^B_\pm\frac{1}{\pi}\ln \frac{1}{r^{B}_\pm}+U^0_\pm+O(r^{B}_\pm)
\qquad \mbox{ as }r^{B}_\pm:=((x- p_{\pm})^2+(y-1-L^{\eps})^2)^{1/2}\rightarrow0^+,
\]
where $U^0_\pm$ are constants. Note that $u^0_\pm$ is singular at $B^\eps_{\pm}$. Integrating by parts in 
\[
\begin{array}{l}
\dsp0=\int_{\Pi^{\eps}_{\pm,\rho}}(\Delta u^0_\pm +\omega^2u^0_\pm)(e^{+i\omega (y-1-L^{\eps})} + e^{-i\omega(y-1-L^{\eps})})\,dxdy\\\dsp-\int_{\Pi^{\eps}_{\pm,\rho}}u^0_\pm\,(\Delta (e^{+i\omega (y-1-L^{\eps})} + e^{-i\omega (y-1-L^{\eps})}) +\omega^2(e^{+i\omega (y-1-L^{\eps})} + e^{-i\omega (y-1-L^{\eps})})\,dxdy,
\end{array}
\]
with $\Pi^\eps_{\pm,\rho}:=\{(x,y)\in\Pi^{\eps}_{\pm}\,,y<\rho\mbox{ and }r^B_{\pm}>1/\rho\}$, and taking the limit $\rho\to+\infty$, we get $2i\om T^0_\pm+2C^B_\pm\cos(\om p_\pm)=0$. From the expressions of $C^B_{\pm}$ (see (\ref{DefInnerBbis})), this gives
\begin{equation}\label{equation2}
T^0_\pm=i(-1)^{1+m_{\pm}}a_\pm\cos(\om p_\pm).
\end{equation}
Matching the constant behaviour inside $\Pi_0$, we get
\[
U^0_{0\pm}=C^A_{\pm}\,\pi^{-1}\ln\eps+c^A_{\pm}=-C^A_{\pm}\,\pi^{-1}|\ln\eps|+c^A_{\pm}.
\]
This sets the value of $c^A_{\pm}$. However $U^0_{0\pm}$ depends on $a_{\pm}$ and we have to explicit this dependence. For $u^0_0$, we have the decomposition
\begin{equation}\label{SecondDecompo}
u^0_0=\mrm{w}^+_h+\mrm{w}^-_h+C^A_+\gamma_++C^A_-\gamma_-
\end{equation}
where $\gamma_{\pm}$ are the outgoing functions such that
\begin{equation}\label{DefGamma}
\begin{array}{|rcll}
\Delta \gamma_{\pm}+\om^2\gamma_{\pm}&=&0&\mbox{ in }\Pi_0\\
\partial_\nu\gamma_{\pm}&=&\delta_{A_{\pm}}&\mbox{ on }\partial\Pi_0.
\end{array}
\end{equation}
Here $\delta_{A_{\pm}}$ stands for the Dirac delta function at $A_{\pm}$. Denote by $\Gamma_{\pm}$ the constant behaviour of $\gamma_{\pm}$ at $A_{\pm}$, that is the constant such that $\gamma_{\pm}$ behaves as 
\[
\gamma_{\pm}(x,y)= \frac{1}{\pi}\ln \frac{1}{r^{A}_\pm}+\Gamma_{\pm}+O(r^{A}_\pm)\qquad \mbox{ when }r^{A}_\pm=((x- p_{\pm})^2+(y-1)^2)^{1/2}\rightarrow0^+.
\]
In Lemma \ref{lemmaRelConstants} below, we will prove that the constant behaviours of $\gamma_{\pm}$ at $A_{\mp}$ are equal. We denote by $\tilde{\Gamma}=\gamma_{+}(A_-)=\gamma_{-}(A_+)$ the value of this coupling constant. Then from (\ref{SecondDecompo}), we derive
\[
U^0_{0\pm}=2\cos(\om p_\pm)+\om(a_\pm\Gamma_\pm+a_\mp\tilde{\Gamma}).
\]
Matching the constant behaviour at $A_{\pm}$ inside the thin strips, we obtain
\begin{equation}\label{BoundaryCondition1}
\begin{array}{lcl}
v_{\pm}^0(1)&=&C^A_{\pm}\,C_{\Xi}+c^A_{\pm} = U^0_{0\pm}+C^A_{\pm}\,(\pi^{-1}|\ln\eps|+C_{\Xi})\\[5pt]
 &=& 2\cos(\om p_\pm)+a_\mp\om\tilde{\Gamma}+a_{\pm}\om\,(\pi^{-1}|\ln\eps|+C_{\Xi}+\Gamma_\pm).
\end{array}
\end{equation}
Now, matching the constant behaviour inside $\Pi^{\eps}_\pm$, we get
\[
U^0_{\pm}=C^B_{\pm}\,\pi^{-1}\ln\eps+c^B_{\pm}=-C^B_{\pm}\,\pi^{-1}|\ln\eps|+c^B_{\pm}.
\]
This sets the value of $c^B_{\pm}$. However $U^0_{\pm}$ depends on $a_{\pm}$ and we have to explicit this dependence. For $u^0_\pm$, we have the decomposition
\[
u^0_\pm(x,y)=C^B_\pm g(x-p_{\pm},y+1+L^{\eps})
\]
where $g$ is the outgoing function such that
\begin{equation}\label{Defg}
\begin{array}{|rcll}
\Delta g+\om^2g&=&0&\mbox{ in } \Pi:=(-1/2;1/2)\times(0;+\infty) \\
\partial_\nu g&=&\delta_{O}&\mbox{ on }\partial\Pi.
\end{array}
\end{equation}
Here $\delta_{O}$ stands for the Dirac delta function at $O$. Denote by $G$ the constant behaviour of $g$ at $O$, that is the constant such that $g$ behaves as 
\[
g(x,y)= \frac{1}{\pi}\ln \frac{1}{r}+G+O(r)\qquad \mbox{ when }r=(x^2+y^2)^{1/2}\rightarrow0^+.
\]
Then we have
\[
U^0_{\pm}=C^B_\pm G=(-1)^{1+m_{\pm}}a_{\pm}\omega G.
\]
Matching the constant behaviour at $B^\eps_{\pm}$ inside the thin strips, we obtain
\begin{equation}\label{BoundaryCondition2}
\begin{array}{lcl}
v_{\pm}^0(1+L_{\pm})&=&C^B_{\pm}\,(L'_{\pm}+C_{\Xi})+c^B_{\pm} = U^0_{\pm}+C^B_{\pm}\,(\pi^{-1}|\ln\eps|+L'_{\pm}+C_{\Xi})\\[5pt]
 &=& (-1)^{1+m_{\pm}}a_{\pm}\omega \,(\pi^{-1}|\ln\eps|+L'_{\pm}+C_{\Xi}+G).
\end{array}
\end{equation}
Writing the compatibility condition so that the problem (\ref{Pb1D1}) supplemented with the boundary conditions (\ref{BoundaryCondition1})--(\ref{BoundaryCondition2}) admits a non zero solution, we get
\[
v_{\pm}^0\partial_y\textbf{v}|_1-v_{\pm}^0\partial_y\textbf{v}|_{1+L_{\pm}}-(\textbf{v}\partial_y v_{\pm}^0|_1-\textbf{v}\partial_y v_{\pm}^0|_{1+L_{\pm}})=0.
\]
Since $\textbf{v}(1)=\textbf{v}(1+L_{\pm})=0$, we obtain
\[
v_{\pm}^0(1)+(-1)^{1+m_{\pm}}v_{\pm}^0(1+L_{\pm})=0.
\]
This gives the relations
\begin{equation}\label{equation3}
2\cos(\om p_\pm)+a_\mp\om\tilde{\Gamma}+a_{\pm}\om\,(2\pi^{-1}|\ln\eps|+L'_{\pm}+2C_{\Xi}+\Gamma_\pm+G)=0.
\end{equation}
Below, we will prove that $\Im m\,(\om \Gamma_{\pm})=|\cos(\om p_{\pm})|^2$, $\Im m\,(\om G)=1$ (Lemma \ref{lemmaRelConstants}) and $C_{\Xi}\in\R$ (Lemma \ref{LemmaCReal}). Therefore \eqref{equation3} writes equivalently
\begin{equation}\label{equation4}
a_{\pm}(\beta_{\pm}+i(1+|\cos(\om p_{\pm})|^2))+a_\mp\om\tilde{\Gamma}=-2\cos(\om p_\pm)
\end{equation}
where $\beta_{\pm}$ are the real valued quantities such that
\begin{equation}\label{defBetaP}
\beta_{\pm}:=\om\,(2\pi^{-1}|\ln\eps|+L'_{\pm}+2C_{\Xi}+\Re e\,\Gamma_\pm+\Re e\,G).
\end{equation}
Identities (\ref{equation4}) form a system of two equations whose unknowns are $a_{\pm}$. Solving it, we get
\begin{equation}\label{Defapm}
a_\pm=\cfrac{2\cos(\om p_{\mp})\om\tilde{\Gamma}-2(\beta_{\mp}+i(1+|\cos(\om p_{\mp})|^2))\cos(\om p_{\pm})}{(\beta_{+}+i(1+|\cos(\om p_{+})|^2))(\beta_{-}+i(1+|\cos(\om p_{-})|^2))-\om^2\tilde{\Gamma}^2}\,.
\end{equation}
And from (\ref{equation1}), (\ref{equation2}), we obtain explicit expressions for $R^0$, $T^0_{\pm}$. This ends the asymptotic analysis. To sum up, when $\eps$ tends to zero, we have obtained the following expansions
\begin{equation}\label{AsymptoFinalResults}
\fbox{$\begin{array}{rcl}
u^{\eps}(x,y)&=&\mrm{w}^+_h(x,y)+\mrm{w}^-_h(x,y)+a_+\om\gamma_+(x,y)+a_-\om\gamma_-(x,y)+\dots \mbox{ in }\Pi_0,\\[4pt]
u^{\eps}(x,y)&=&(-1)^{1+m_{\pm}}a_{\pm}\omega g(x-p_\pm,y+1+L^{\eps})
+\dots \mbox{ in }\Pi^\eps_\pm, \\[4pt]
u^{\eps}(x,y)&=&\eps^{-1}a_\pm\sin(\om(y-1))+\dots \mbox{ in }\mathcal{S}^\eps_\pm,\\[4pt]
\multicolumn {3}{l}{R^\eps=1+i(a_+\cos(\om p_+)+a_-\cos(\om p_-))+\dots,\qquad T^\eps_\pm=i(-1)^{1+m_{\pm}}a_\pm\cos(\om p_\pm)+\dots,}
\end{array}$}
\end{equation}
where $a_\pm$ are given by (\ref{Defapm}). Here, the functions $\gamma_\pm$, $g$ are respectively introduced in (\ref{DefGamma}), (\ref{Defg}). Note in particular that when $a_\pm\ne0$, the amplitude of the field blows up in the thin slits as $\eps$ tends to zero.

\section{Analysis of the results}\label{AnalysisResults}
In this section, we explain how to use the asymptotic results (\ref{AsymptoFinalResults}) to exhibit settings where the waveguide $\Om^{\eps}$ acts as an energy distributor. The degrees of freedom we can play with are $L'_{\pm}$ and $p_\pm$, that is the lengths and the abscissa of the thin strips. \\
\newline
$\star$ When $p_{\pm}$ are chosen such that $\cos(\om p_{\pm})=0$, from (\ref{Defapm}) we get $a_{\pm}=0$. This implies $R^0=1$ and $T^0_{\pm}=0$. In this case, the energy brought to the system is almost completely backscattered and the thin strips have almost no influence on the incident field. Definitely, this is not an acoustic distributor. Roughly speaking, in this situation what happens is that the resonant eigenfunctions  associated with complex resonances existing due to the presence of the thin slits are not excited.\\
\newline
$\star$ When $p_{\pm}$ are chosen such that $\cos(\om p_{\pm})=1$, we find 
\[
a_{\pm}=\cfrac{2\om\tilde{\Gamma}-2(\beta_{\mp}+2i)}{(\beta_++2i)(\beta_-+2i)-\om^2\tilde{\Gamma}^2}\,.
\]
Then, we have 
\[
R^0=1+i(a_++a_-)=\cfrac{\beta_+\beta_-+4-\om^2\tilde{\Gamma}^2+4i\om \tilde{\Gamma}}{(\beta_++2i)(\beta_-+2i)-\om^2\tilde{\Gamma}^2}\,.
\]
and 
\[
T^0_\pm=i(-1)^{1+m_{\pm}}\cfrac{2\om\tilde{\Gamma}-2(\beta_{\mp}+2i)}{(\beta_++2i)(\beta_-+2i)-\om^2\tilde{\Gamma}^2}\,.
\]
According to Lemma \ref{lemmaRelConstants} below, we have $\om\tilde{\Gamma}=\eta+i$ for a certain $\eta\in\R$ which characterizes the coupling between the two strips. With this notation, we find
\begin{equation}\label{DefEquationEta}
R^0=\cfrac{\beta_+\beta_-+1-\eta^2+2i\eta}{\beta_+\beta_-+2i(\beta_++\beta_-)-3-\eta^2-2i\eta},\qquad T^0_\pm=\cfrac{2i(-1)^{m_{\pm}}(\beta_{\mp}+i-\eta)}{\beta_+\beta_-+2i(\beta_++\beta_-)-3-\eta^2-2i\eta}\,.
\end{equation}
Notice that the denominator in the expressions for $R^0$, $T^0_{\pm}$ cannot vanish. Indeed it vanishes if and only if $\beta_++\beta_-=\eta$ and  $\beta_+\beta_-=3+\eta^2$. One can verify that this cannot occur for $(\beta_+,\beta_-)\in\R^2$.\\
\newline
If $\eta=0$ was zero, then we would have $R^0=\mathcal{R}^0$, $T^0_\pm=\mathcal{T}^0_\pm$ with 
\begin{equation}\label{CoefScaDoubleApp}
\mathcal{R}^0=\cfrac{\beta_+\beta_-+1}{\beta_+\beta_-+2i(\beta_++\beta_-)-3},\qquad \mathcal{T}^0_\pm=\cfrac{2i(-1)^{m_{\pm}}(\beta_{\mp}+i)}{\beta_+\beta_-+2i(\beta_++\beta_-)-3}\,.
\end{equation}
In particular, we would have $R^0=0$ for all pairs $(\beta_+,\beta_-)\in\R^2$ such that 
\begin{equation}\label{ConditionBeta}
\beta_+\beta-=-1,
\end{equation}
and then $|T_+^0(\beta_+,\beta_-)|^2+|T_-^0(\beta_+,\beta_-)|^2=1$ (conservation of energy) as well as 
\begin{equation}\label{DefRatio}
\cfrac{|T_+^0(\beta_+,\beta_-)|}{|T_-^0(\beta_+,\beta_-)|}=\sqrt{\cfrac{\beta_-^2+1}{\beta_+^2+1}}=\sqrt{\cfrac{1/\beta_+^2+1}{\beta_+^2+1}}\,.
\end{equation}
Observing in (\ref{defBetaP}) that $\beta_+$ varies in $\R$ when $L'_\pm$ runs in $\R$, we see that the ratio (\ref{DefRatio}) could take any value in $(0;+\infty)$.\\
\newline
However, the coupling constant $\eta$ cannot be chosen as we wish because we have already set $p_\pm$ to impose $\cos(\om p_\pm)=1$. In Lemma \ref{lemmaRelConstants} below, we will prove that the function $e_{\pm}:=\gamma_{\pm}-s_{\pm}\mrm{w}^-_h$, with $s_{\pm}=i\cos(\om p_\pm)/\om$, is real and exponentially decaying at infinity. As a consequence, we infer that for $p_-<p_+$ with $|p_+-p_-|$ large, we have 
\[
\om\tilde{\Gamma}=\om\gamma_{+}(p_-,1)\approx \om s_{+}\mrm{w}^-_h(p_-)=i\cos(\om p_+)e^{-i\om p_-}.
\]
With the above choice for $p_{\pm}$ as well as the relation (\ref{ConditionBeta}) for $\beta_{\pm}$, which translates into a condition relating $L'_{+}$ to $L'_{-}$, we have $\eta=\Re e\,(\om\tilde{\Gamma})\approx 0$. From the previous analysis, we deduce that
\begin{equation}\label{FinalDistributor}
R^0\approx 0\qquad\mbox{ together with }\qquad  \cfrac{|T_+^0(\beta_+,\beta_-)|}{|T_-^0(\beta_+,\beta_-)|}\approx \sqrt{\cfrac{\beta_-^2+1}{\beta_+^2+1}}=\sqrt{\cfrac{1/\beta_+^2+1}{\beta_+^2+1}}\,.
\end{equation}
Thus, we can get almost complete transmission and varying $L'_+$, and so in practice, varying slightly $L^{\eps}_+$ around $\pi m_+/\om$, we can control the ratio of energy transmitted respectively in $\Pi_+^\eps$ and in $\Pi_-^\eps$. We emphasize that the tuning becomes more or more subtle as $\eps$ gets smaller. More precisely, for a fixed value of $\beta_\pm$, according to (\ref{defBetaP}), the corresponding value of $L'_{\pm}$ is equal to $-2\pi^{-1}|\ln\eps|+\beta_{\pm}/\om-2C_{\Xi}-\Re e\,\Gamma_\pm-\Re e\,G$ (note in particular that it is negative for $\eps$ small enough). As a consequence, 
\[
\begin{array}{rcl}
L^{\eps}_{\pm}&=& \pi m_{\pm}/\om+\eps L'_{\pm}\\[3pt]
 &= &  \pi m_{\pm}/\om+\eps (-2\pi^{-1}|\ln\eps|+\beta_{\pm}/\om-2C_{\Xi}-\Re e\,\Gamma_\pm-\Re e\,G)
\end{array}
\]
converges to $(\pi m_{\pm}/\om)^-$ as $\eps$ tends to zero.

\begin{remark}
In the method, we tune the lengths of the slits around the resonance lengths. One may imagine to play with other parameters. For example one could work with slits of variable width
\[
\mathcal{S}^{\eps}_{\pm}=\{(x,y)\in\R^2\,|\,y\in[1;1+L_{\pm}],\,|x-p_{\pm}|< \eps H_\pm(y)\},
\]
where $H_\pm$ are smooth profile functions, and then perturb $H_\pm$. Or we could also perturb slightly the position of the slits, that is $p_{\pm}$.
\end{remark}

\section{Numerics}\label{SectionNumerics}

In this section, we illustrate the results that we have obtained above. We work with 
\[
\om=0.8\pi,\qquad p_-=-2.5,\qquad p_+=0,\qquad \eps=0.05.
\]
Observe that in this case indeed we have $\cos(\om p_+)=\cos(\om p_-)=1$. We compute numerically\footnote{The code, written in \texttt{Freefem++}, can be found at the following address 
\url{http://www.cmap.polytechnique.fr/~chesnel/Documents/EnergyDistributor.edp}.} the scattering solution $u^{\eps}$ in (\ref{MainPb}). To proceed, we use a $\mrm{P2}$ finite element method in a truncated geometry. On the artificial boundary created by the truncation, a Dirichlet-to-Neumann operator with 15 terms serves as a transparent condition (see more details for example in \cite{Gold82,HaPG98,BoLe11}). Once we have computed $u^{\eps}$, we get the scattering coefficients $R^{\eps}$, $T^{\eps}_\pm$  in the representation (\ref{Field}). \\
\newline
In the first line of Figure \ref{CompaAsympto}, we represent $R^{\eps}$, $T^{\eps}_+$, $T^{\eps}_-$ as functions of $(L^{\eps}_+,L^{\eps}_+)$. In the second line of Figure \ref{CompaAsympto}, we display $\mathcal{R}^{0}$, $\mathcal{T}^{0}_+$, $\mathcal{T}^{0}_-$ (see (\ref{CoefScaDoubleApp})) as functions of $(\beta_+,\beta_-)$. We observe that the true scattering coefficients and their asymptotic approximations coincide very well. From these results, we extract the curve $(L^{\eps}_+,L^{\eps}_-)$ where  $|R^{\eps}(L^{\eps}_+,L^{\eps}_-)|$ is minimum (see Figure \ref{FigExtracted}). We note that indeed we can have almost no reflection. Additionally, we show the quantities $|T^{\eps}_{\pm}(L^{\eps}_+,L^{\eps}_-)|$ on this curve. As predicted, we notice that we can indeed control the ratio of energy transmitted in the channels $\Pi^\eps_+$, $\Pi^\eps_-$. Finally, in Figure \ref{Fields}, we represent the field $u^\eps$ in different geometries: one without particular tuning of the lengths of the slits, one where $R^{\eps}\approx0$, $|T_+^{\eps}|\approx1$ and $T_-^{\eps}\approx0$ and one where $R^{\eps}\approx0$, $T_+^{\eps}\approx0$ and $|T_-^{\eps}|\approx1$. Our device indeed can act as an acoustic energy distributor.

\begin{remark}
We emphasize that the asymptotic approximation gets more and more accurate as $\eps$ tends to zero. Numerically however, we observe a quite good energy transmission for $\eps$ not that small. This is interesting to confer some robustness to the device with respect to perturbations of the geometry. Indeed, when $\eps$ is very small, the lengths of the slits  must be tuned very precisely to prevent backscattering of energy.
\end{remark}

\begin{figure}[!ht]
\centering
\includegraphics[width=0.32\textwidth]{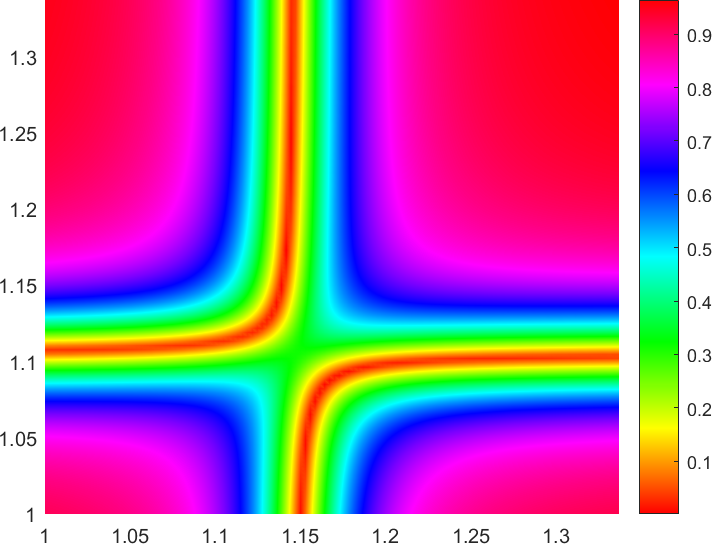}\ \includegraphics[width=0.32\textwidth]{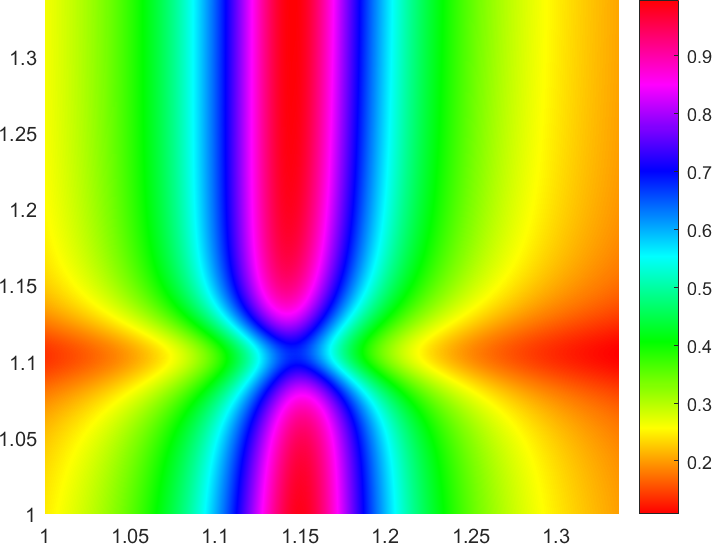}\ \includegraphics[width=0.32\textwidth]{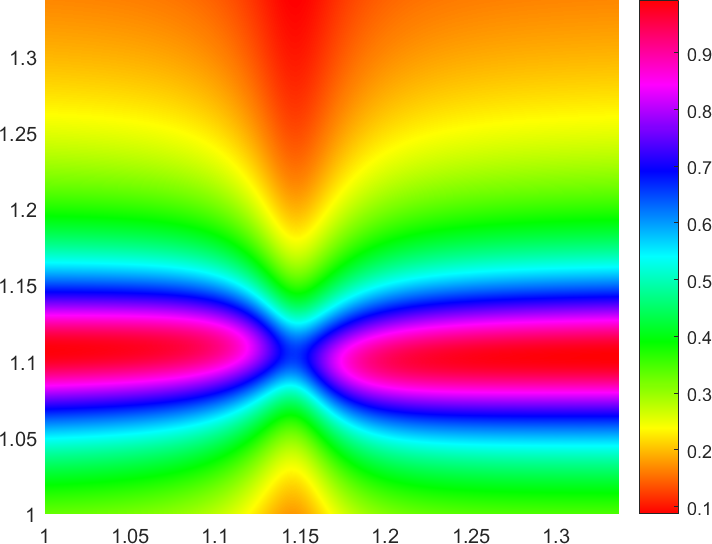}\\
\includegraphics[width=0.32\textwidth]{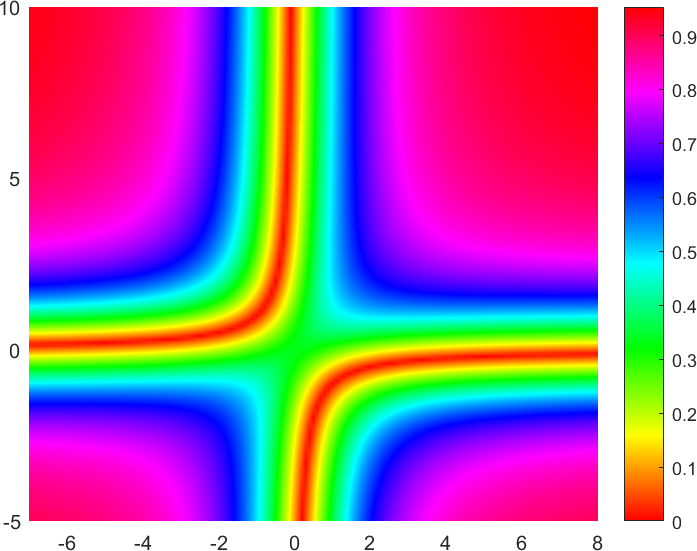}\ \includegraphics[width=0.32\textwidth]{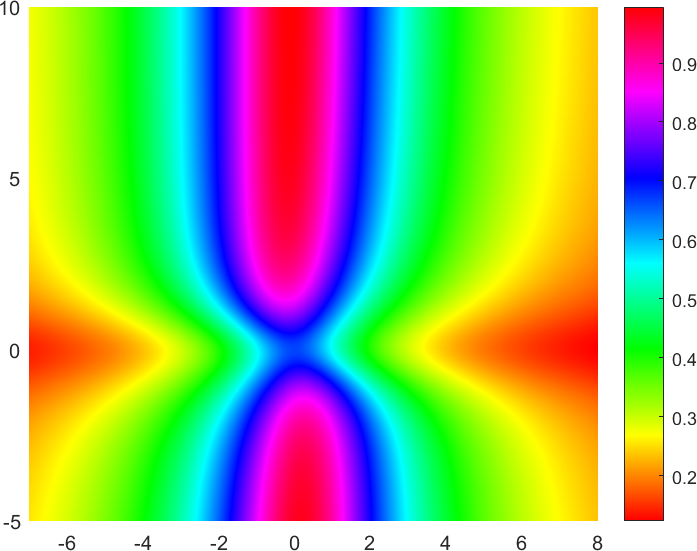}\ \includegraphics[width=0.32\textwidth]{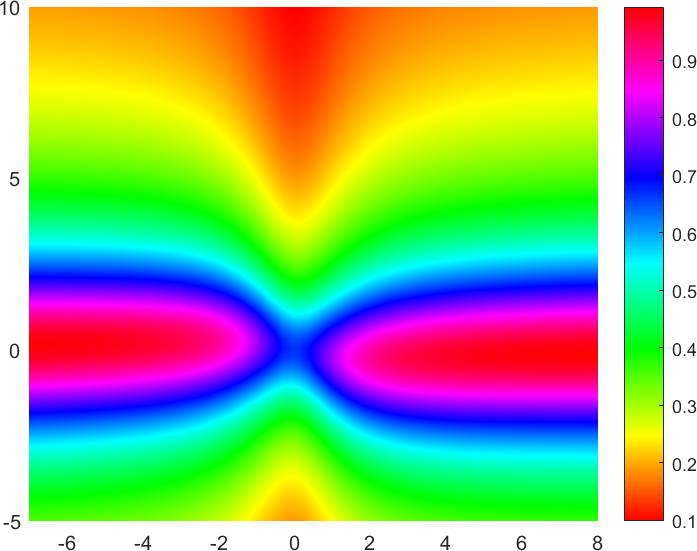}\\
\caption{On the first line, from left to right, we represent $R^{\eps}$, $T^{\eps}_+$, $T^{\eps}_-$ as functions of $(L^{\eps}_+,L^{\eps}_-)$. On the second line, from left to right, we represent $\mathcal{R}^{0}$, $\mathcal{T}^{0}_+$, $\mathcal{T}^{0}_-$ as functions of $(\beta_+,\beta_-)$. The true scattering coefficients and their principal asymptotic approximations coincide very well.\label{CompaAsympto}}
\end{figure}

\begin{figure}[!ht]
\centering
\includegraphics[width=0.8\textwidth]{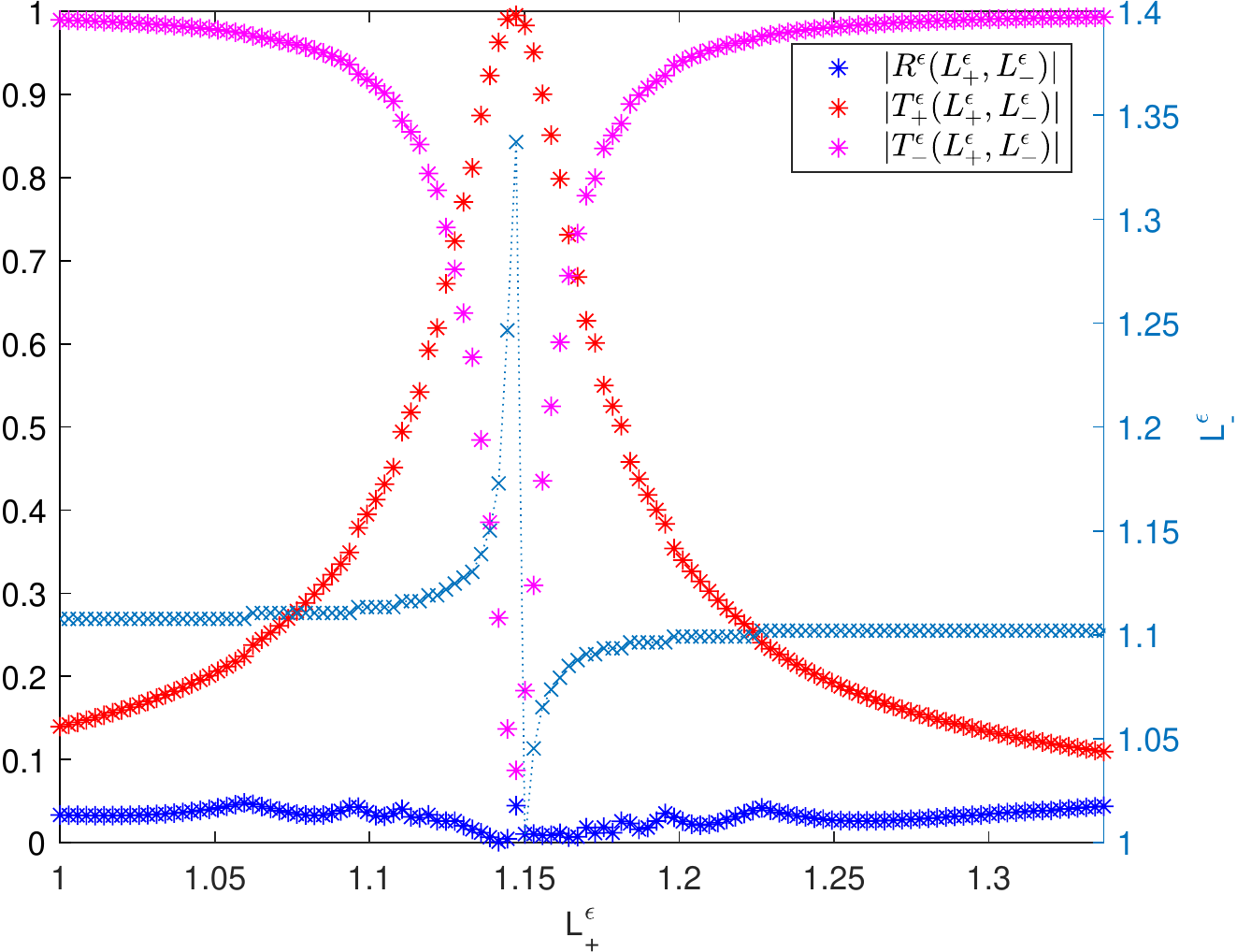}
\caption{From the results of Figure \ref{CompaAsympto}, we extract the curve $(L^{\eps}_+,L^{\eps}_-)$ where  $|R^{\eps}(L^{\eps}_+,L^{\eps}_-)|$ is minimum. Then we display $|T^{\eps}_{\pm}(L^{\eps}_+,L^{\eps}_-)|$ on this curve. \label{FigExtracted}}
\end{figure}

\begin{figure}[!ht]
\centering
\includegraphics[width=0.48\textwidth]{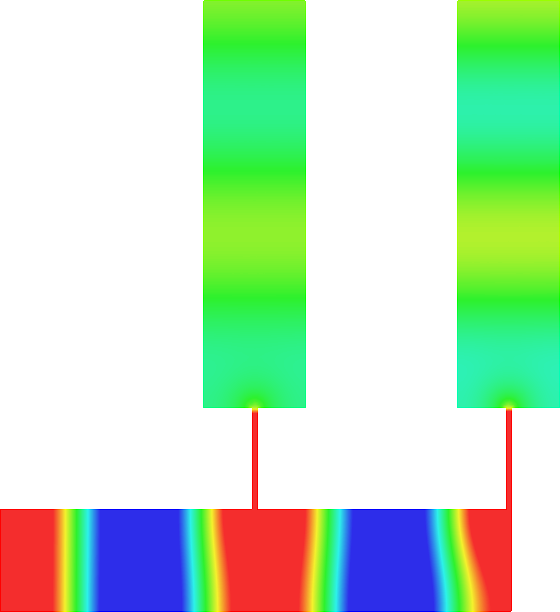}\\[30pt]
\includegraphics[width=0.48\textwidth]{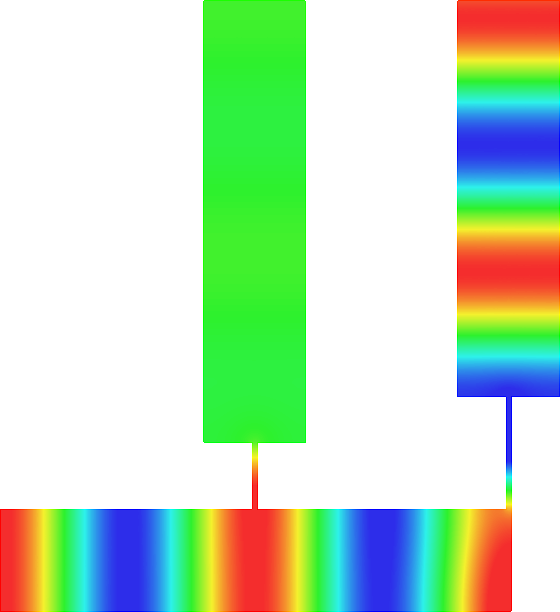}\quad\includegraphics[width=0.48\textwidth]{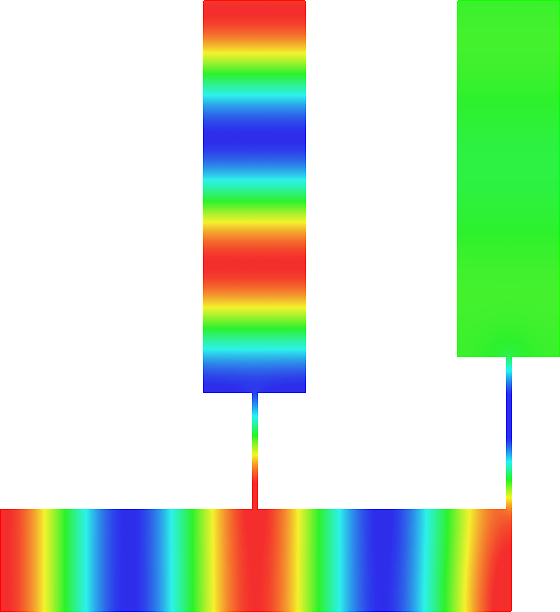}
\caption{First line: without particular tuning of the lengths of the slits, the energy of the incoming wave is almost completely backscattered. Second line, left: setting where $R^{\eps}\approx0$, $T_+^{\eps}\approx0$ and $|T_-^{\eps}|\approx1$. Second line, right: setting where $R^{\eps}\approx0$, $|T_+^{\eps}|\approx1$ and $T_-^{\eps}\approx0$. \label{Fields}}
\end{figure}

\newpage
\clearpage

\section{Concluding remarks}\label{SectionConclusion}

\begin{figure}[!ht]
\centering
\begin{tikzpicture}[scale=1.2]
\draw[fill=gray!30,draw=none](-2,0) rectangle (0,1);
\draw[fill=gray!30,draw=none](-2,0) rectangle (0,1);
\draw (0,0)--(0,1);
\draw[fill=gray!30,draw=none](-0.2,0) rectangle (1,0.1);
\draw[fill=gray!30,draw=none](-0.2,0.9) rectangle (1,1);
\draw (0,0.1)--(1,0.1);
\draw (0,0.9)--(1,0.9);
\draw (-2,1)--(1,1);
\draw (-2,0)--(1,0);
\draw[fill=gray!30,draw=none](2.9,0.4) rectangle (0.9,-0.6);
\draw[fill=gray!30,draw=none](3,0.5) rectangle (1,1.5);
\draw (3,0.5)--(1,0.5);
\draw (3,1.5)--(1,1.5);
\draw (2.9,0.4)--(0.9,0.4);
\draw (2.9,-0.6)--(0.9,-0.6);
\draw (1,0.5)--(1,0.9);
\draw (1,1.5)--(1,1);
\draw (0.9,-0.6)--(0.9,0);
\draw (0.9,0.1)--(0.9,0.4);
\draw[dashed] (-2,1)--(-2.5,1);
\draw[dashed] (-2,0)--(-2.5,0);
\draw[dashed] (3,0.5)--(3.5,0.5);
\draw[dashed] (3,1.5)--(3.5,1.5);
\draw[dashed] (2.9,0.4)--(3.4,0.4);
\draw[dashed] (2.9,-0.6)--(3.4,-0.6);
\end{tikzpicture}
\caption{A device which fails to behave as an acoustic distributor. \label{InoperantDomain}} 
\end{figure}
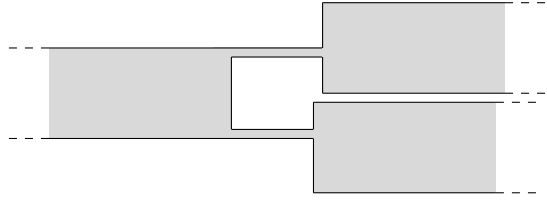

\noindent $i)$  From this analysis, a natural idea is to try to work in the waveguide represented in Figure \ref{InoperantDomain} with two slits at the end of the input channel. In this case, the asymptotic expansion is similar to the one above and we have $\cos(\om p_{\pm})=1$ in (\ref{Defapm}), (\ref{AsymptoFinalResults}). But the problem is that in this geometry we cannot reduce the coupling constant $\eta$ (see (\ref{DefEquationEta})) as we wish, which is a very important point. As a consequence, this device is not interesting for our purpose.\\
\newline
$ii)$ We considered straight vertical slits to simplify the presentation and to limit the complexity of the notation. We could have worked similarly  with slits coinciding at the limit $\eps\to0$ with some smooth curves, we would have observed the same phenomena. What matters is the lengths of the slits. Additionally, the orientation of the slits (if they are not vertical) plays no major role.\\ 
\newline
$iii)$ A priori the approach presented here can also be used to construct an acoustic energy distributor with one input/output channel and more than two output channels. The question of working at higher wavenumber, so that several modes can propagate, seems less simple to address. Indeed, in this case, there are more than three scattering coefficients coming into play, and even though they are related by some structure (due to conservation of energy, reciprocity relations, ...), it appears difficult 
to control them with only two slits. A natural idea is to work with more slits. But then the coupling effects, whose dependence with respect to the geometry is not very explicit, become hard to handle.\\
\newline
$iv)$ What was done in 2D here could be adapted in 3D. The asymptotic expansion would be different, in particular because the equivalent of the function $Y^1$ introduced in (\ref{PolyGrowth}) has a different behaviour at infinity in 3D, but the methodology would be the same (see \cite{NaCh21a,NaCh21b} for related works). \\
\newline
$v)$ On the contrary, the approach proposed in this article is very specific to Neumann Boundary Conditions (BCs) and cannot be adapted for Dirichlet BCs (quantum waveguides). Indeed, with Dirichlet BCs nothing passes through the thin slits and almost all the incoming energy is backscattered. Therefore, to design an energy distributor with Dirichlet BCs, it is necessary to find a different idea. The problem seems even more open with other types of BCs, for example with BCs modelling realistic materials with losses.

\section*{Appendix: auxiliary results}\label{Appendix}

\begin{lemma}\label{LemmaCReal}
The constant $C_{\Xi}$ is real.
\end{lemma}
\begin{proof}
Let $Y^1$ be the function introduced in (\ref{PolyGrowth}). For all $\rho>0$, we have 
\[
0=\int_{\Xi_\rho} (Y^1-\overline{Y^1})\Delta Y^1-\Delta(Y^1-\overline{Y^1})Y^1\,d\xi_xd\xi_y
\]
with $\Xi_\rho:=\{(\xi_x,\xi_y)\in\Xi,\,\xi_y<0\mbox{ and }|\xi|<\rho\}\cup \{(\xi_x,\xi_y)\in(-1/2;1/2)\times[0;\rho)\}$. Integrating by parts and taking the limit $\rho\to+\infty$, we get $C_\Xi-\overline{C_\Xi}=0$. This shows that $C_\Xi$ is real.
\end{proof}

\begin{lemma}\label{lemmaRelConstants}
The constant behaviours of $\gamma_{\pm}$ at $A_{\mp}$ are equal. We denote by $\tilde{\Gamma}\in\Cplx$ this constant. We have $\Im m\,(\om\tilde{\Gamma})=\cos(\om p_+)\cos(\om p_-)$. The constants $\Gamma_{\pm}$ are such that $\Im m\,(\om\Gamma_{\pm})=|\cos(\om p_\pm)|^2$. 
\end{lemma}
\begin{proof}
Since the functions $\gamma_{\pm}$ are outgoing at infinity, we have the expansions
\[
\gamma_{\pm}=s_{\pm} \mrm{w}^-_h+\tilde{\gamma}_{\pm}
\]
where $s_{\pm}\in\Cplx$ and $\tilde{\gamma}_{\pm}$ are exponentially decaying at infinity. We have
\[
0=\int_{\Pi_{0,\rho}}(\Delta \gamma_+ +\omega^2\gamma_+)\gamma_--\gamma_+\,(\Delta \gamma_- +\omega^2\gamma_-)\,dxdy,
\]
with $\Pi_{0,\rho}=\{(x,y)\in\Pi_0\,,x>-\rho\mbox{ and }r^A_{\pm}>1/\rho\}$. Integrating by parts and taking the limit $\rho\to+\infty$, we find that the constant behaviours of $\gamma_{\pm}$ at $A_{\mp}$ are equal. On the other hand, integrating by parts in
\[
0=\int_{\Pi_{0,\rho}}(\Delta \gamma_{\pm} +\omega^2\gamma_{\pm})(e^{+i\omega x} + e^{-i\omega x})-\gamma_{\pm}\,(\Delta (e^{+i\omega x} + e^{-i\omega x}) +\omega^2(e^{+i\omega x} + e^{-i\omega x}))\,dxdy,
\]
and taking the limit $\rho\to+\infty$, we obtain 
\begin{equation}\label{relationCoefSca}
s_{\pm}=i\cos(\om p_\pm)/\om.
\end{equation}
Then one can verify that the function $e_{\pm}:=\gamma_{\pm}-s_{\pm}\mrm{w}^-_h$ is real. Indeed $e_{\pm}-\overline{e_{\pm}}$ is exponentially decaying and solves the homogeneous problem. We deduce that $\Im m\,(\om\tilde{\Gamma})=\cos(\om p_+)\cos(\om p_-)$. Finally, integrating by parts in
\[
0=\int_{\Pi_{0,\rho}}(\Delta \gamma_{\pm} +\omega^2\gamma_{\pm})\overline{\gamma_{\pm}}-\gamma_{\pm}\,(\Delta \overline{\gamma_{\pm}} +\omega^2\overline{\gamma_{\pm}})\,dxdy,
\]
and taking again the limit $\rho\to+\infty$, we obtain $2i\om|s_{\pm}|^2-2\Im m\,\Gamma_{\pm}=0$. From (\ref{relationCoefSca}), this yields the desired result.
\end{proof}

\section*{Acknowledgments} 
The work of S.A. Nazarov was supported by the Ministry of Science and Higher Education of Russian Federation within the framework of the Russian State Assignment under contract No.  FFNF-2021-0006.

\bibliography{Bibli}
\bibliographystyle{plain}
\end{document}